\documentclass[11pt]{article}
\usepackage{amsmath,amssymb,amsthm,graphicx,subfigure,float,url}
\usepackage[colorlinks=true,citecolor=red,filecolor=red,urlcolor=magenta]{hyperref}
\usepackage{pdfsync}

\usepackage{refcheck}

\newcommand{\Op}{\mathrm{Op}}
\newcommand{\QL}{\mathcal{Q}}

\newcommand{\om} {\omega}

\newcommand{\ep} {\varepsilon}
\newcommand{\al} {\alpha}

\newcommand{\la} {\lambda}
\newcommand{\La} {\Lambda}

\newcommand{\N} {\mathbb{N}}

\renewcommand{\geq}{\geqslant}
\renewcommand{\leq}{\leqslant}
\newtheorem{theorem}{Theorem}  

\newtheorem{corollary}{Corollary}
\newtheorem{definition}{Definition}
\newtheorem{lemma}{Lemma}

\theoremstyle{definition}\newtheorem{remark}{Remark}

\title{Observability for generalized Schr\"odinger equations and quantum limits on product manifolds}

\author{Emmanuel Humbert\footnote{Institut Denis Poisson, UFR Sciences et Technologie, Facult\'e Fran\c cois Rabelais, Parc de Grandmont, 37200 Tours, France (\texttt{emmanuel.humbert@lmpt.univ-tours.fr}).}
\and Yannick Privat\footnote{IRMA, Universit\'e de Strasbourg, CNRS UMR 7501, 7 rue Ren\'e Descartes, 67084 Strasbourg, France (\texttt{yannick.privat@unistra.fr}).}
	\and Emmanuel Tr\'elat\footnote{Sorbonne Universit\'e, CNRS, Universit\'e de Paris, Inria, Laboratoire Jacques-Louis Lions (LJLL), F-75005 Paris, France (\texttt{emmanuel.trelat@sorbonne-universite.fr}).}}

\begin{document}

\maketitle

\begin{abstract}
Given a closed product Riemannian manifold $N=M \times M'$ equipped with the product Riemannian metric $g=h + h'$, we explore the observability properties for the generalized Schr\"odinger equation $i \partial_t u = F(\triangle_g) u$, where $\triangle_g$ is the Laplace-Beltrami operator on $N$ and $F : [0,+\infty) \to [0,+\infty)$ is an increasing function. In this note, we prove observability in finite time on any open subset $\om$ satisfying the so-called Vertical Geometric Control Condition, stipulating that any vertical geodesic meets $\om$, under the additional assumption that the spectrum of $F(\triangle_g)$ satisfies a gap condition.
A first consequence is that observability on $\om$ for the Schr\"odinger equation is a strictly weaker property than the usual Geometric Control Condition on any product of spheres. 
A second consequence is that the Dirac measure along any geodesic of $N$ is never a quantum limit. 
\end{abstract}

\section{Introduction and main results}
Let $(M,h)$, $(M',h')$ be closed Riemannian manifolds, and let $\triangle_h$ and $\triangle_{h'}$ be their respective (nonnegative) Laplace-Beltrami operators. We consider the Riemannian product manifold $(N,g)$ defined by $N=M\times M'$ and $g=h+h'$. Let $\triangle_g=\triangle_h\otimes\triangle_{h'}$ be the corresponding Laplace-Beltrami operator on $N$. Let $F : [0,+\infty) \to [0,+\infty)$ be an arbitrary increasing function. 
We consider the generalized Schr\"odinger equation 
\begin{equation} \label{schrod}
 i \partial_t u = F(\triangle_g) u
\end{equation}
on $M$, and we are interested in finding characterizations of the observability property for \eqref{schrod} on any open subset $\om\subset N$.

We denote by $0=\mu_0\leq\mu_1\leq\cdots\leq\mu_k\leq\cdots$ (resp., $0=\mu'_0\leq\mu'_1\leq\cdots\leq\mu'_k\leq\cdots$) the eigenvalues of $\triangle_h$ (resp., of $\triangle_{h'}$), associated with a Hilbert eigenbasis $(\phi_k)_{k\in\N}$ of $L^2(M)$ (resp., $(\phi'_k)_{k\in\N}$ of $L^2(M')$). We also denote by $0=\lambda_0<\lambda_1<\cdots<\lambda_k<\cdots$ (resp., $0=\lambda'_0<\lambda'_1<\cdots<\lambda'_k<\cdots$) its distinct eigenvalues.
There exists an increasing sequence $(\alpha_k)_{k\in\N}$ (resp., $(\al'_k)_{k\in\N}$) such that for $j=\al_k,\ldots, \al_{k+1}-1$ (resp., $j=\al'_k,\ldots,\al'_{k+1}-1$, $\mu_j = \la_k$ (resp., $\mu'_l = \la_k'$).

Then, $(\phi_j \phi'_k)_{j,k\in\N}$ is an orthonormal basis of $L^2(N)$ of eigenfunctions of $\triangle_g$ associated to the eigenvalues $\mu_j + \mu'_k$. 
The operator $F(\triangle_g)$ is spectrally defined as the linear operator which, restricted to the eigenspace of $\triangle_g$ associated to the eigenvalue $\mu_j + \mu'_k$, is equal to $F(\mu_j + \mu'_k)\,\mathrm{id}$.
 The fact that $F(\triangle_g)$ and $\triangle_g$ have the same eigenspaces comes from the fact that $F$ is  increasing, so that $F(\mu_j+\mu'_k) = F(\mu_{j'} + \mu'_{k'})$ if and only if  $\mu_j+\mu'_k = \mu_{j'} + \mu'_{k'}$.

By the Stone theorem, $(e^{it F(\triangle_g)})_{t\geq 0}$ is a unitary strongly continuous semigroup on $L^2(N)$.
Given any $y\in L^2(N)$, there exists a unique solution $u\in C^0([0,+\infty),L^2(N)) \cap C^1((0,+\infty),H^{-2}(N))$ of \eqref{schrod} such that $u(0)=y$, given by $u(t)=e^{it F(\triangle_g)}y$.

If $F(s)=s$ then \eqref{schrod} is the usual Schr\"odinger equation, and if $F(s) = \sqrt{s}$ then \eqref{schrod} is the half-wave equation. 

We denote by $dx_g$ the Riemannian volume form on $N$.
Given any $T>0$ and any measurable subset $\om$ of $N$, we define the observability constant $C_T(\om)\geq 0$ as the largest constant $C\geq 0$ such that
\begin{equation}\label{obs}
\int_0^T \int_\om \left| e^{it F(\triangle_g)}y  \right|^2 \geq C \|y\|_{L^2(N)}\qquad \forall y\in L^2(N)
\end{equation}
(observability inequality), i.e., 
\begin{equation*}
\begin{split}
C_T(\om)  &= \inf \left\{ \int_0^T \int_\om \left| e^{it F(\triangle_g)}y  \right|^2 \, dx_g\, dt  \ \mid\ y \in L^2(N),\ \|y\|_{L^2(N)} = 1 \right\} \\
&=  \inf \left\{ \int_0^T \int_\om \Big| \sum_{l,m} b_{jk} e^{i t F(\mu_j + \mu_k)} \phi_j \phi'_k \Big|^2 \, dx_g\, dt \ \mid\ b_{jk} \in \ell^2(\mathbb{C}),\ \sum_{j,k=0}^{+\infty} |b_{jk}|^2 = 1 \right\}
\end{split}
\end{equation*}
We say that the {\em {observability property}} is satisfied for \eqref{schrod} on $(\om,T)$ if $C_T(\om) >0$.

\begin{definition} 
A vertical (resp., horizontal) geodesic of $N$ is a geodesic of the form $t \to (x, \gamma(t))$ (resp., $(\gamma(t),x)$) for some $x \in M$ (resp., for some $x\in M'$) and some geodesic $\gamma$ of $M'$ (resp., of $M$). 
\end{definition}

\begin{definition} 
Let $\om\subset N$ and let $T>0$. We say that $(\om,T)$ satisfies the  {\em{Vertical Geometric Control Condition} }(in short, VGCC) if all vertical geodesics meet $\om$ within time $T$, i.e., $\gamma([0,T]) \cap \om \not= \emptyset$. 
\end{definition}

\begin{definition} 
We say that a family $(a_k)_{k\in\N}$ of real numbers satisfies the {\em  gap condition} if there exists a constant $C >0$ such that for all $j,k\in\N$, we have either $a_j=a_k$ or $|a_k - a_l | \geq C$, i.e., if all distinct elements are at a distance of at least $C$ one from each other.
\end{definition}

\begin{theorem} \label{main} 
Let $T>0$ and $\om$ be an open subset of $N$. If $(\om,T)$ satisfies VGCC and if the family $\left( F(\la_j + \la'_k)\right)_{j,k\in\N}$ satisfies the gap condition, then the observability property is satisfied for \eqref{schrod} on $(\om,T)$. 
\end{theorem} 

Let us comment on this theorem and on VGCC.  

Recall that $(\omega,T)$ satisfies the usual {\em{Geometric Control Condition}} (GCC, see \cite{BardosLebeauRauch,lebeau,rauch-taylor}) whenever every geodesic (not necessarily vertical) meets $\om$ within time $T$.
 Let $\om$ be an open subset of $N$ and $T >0$. If $(\om,T)$ satisfies GCC then it also satisfies VGCC.  
There exist examples where $(\om,T)$ satisfies VGCC but not GCC: for every $x \in M$, we define $\om_x := \left( \{x \} \times M'\right) \cap \om$.
Then, $(\om,T)$ satisfies VGCC if and only if $(\om_x,T)$ satisfies GCC on $M'$ for every $x\in M$.
In particular,  we obtain the following examples:
\begin{itemize}
\item Let $(U_i)_{i \in I}$ be an open covering of $M$, and let $(\om_i)_{i\in I}$ be a family of open subsets of $M'$ satisfying GCC within time $T$. Then, setting $\om = \cup_{i \in I} U_i \times \om_i$, $(\om,T)$ satisfies VGCC. In particular, if $\om'$ is an open subset of $M'$ satisfying GCC, then $M \times \om'$ satisfies VGCC. 
\item Let $\gamma$ be a non-vertical geodesic. Given any $\ep>0$, we consider the closed $\ep$-neighborhood of the support $\Gamma$ of $\gamma$  defined by $U_\ep=\{ x \in N \ \mid\ d_g(x,\Gamma) \leq \ep\}$, where $d_g$ is the Riemannian distance on $N$. We set $\om_\ep = N \setminus U_\ep$. Then, for any $T >0$ and any $\ep>0$ small enough, $(\om_\ep,T)$ satisfies VGCC. 

For instance, if $\gamma$ is horizontal, we can choose $\ep < \frac{T}{2}$. For the general case, note that for every $x \in M$, $(\om_\ep)_x$ (with the notations above) is contained in the complement of a small ball in $M'$.  
\end{itemize}

Let us now recall some existing results.   It is well known that, when $\om$ is open, GCC is a sufficient condition for observability of the Schr\"odinger equation (see \cite{lebeau}).
It is also well known that, except for Zoll manifolds, i.e., manifolds whose all geodesics are periodic (see \cite{Macia_2011}), GCC is not a necessary assumption.
An example where the Schr\"odinger is observable on $(\om,T)$ but where $(\om,T)$ does not satisfy GCC is given in \cite{jaffard}: in the flat 2D torus, any non empty open set gives observability in any time $T$. This example has been extended to high dimensions in \cite{Komornik}. We also refer to \cite{ananth} for another example, in the Dirichlet disk.

\begin{remark} 
The spectrum of $\triangle_g^{1/2}$ can never satisfy the gap condition on the product manifold $N$.
\end{remark}

\paragraph{Application to the Schr\"odinger equation.}
We assume that $F(s)=s$ so that \eqref{schrod} is now the usual Schr\"odinger equation. 
Theorem \ref{main} can be applied as soon as the spectrum of $\triangle_g$ satisfies the gap condition. 
This is true for instance when $M$ and $M'$ have an integer spectrum, in particular when $M$ and $M'$ are a finite product of standard spheres. 

\begin{corollary} \label{cor1} 
Assume that the spectrum of $\triangle_g$ satisfies the gap condition. 
Let $T>0$ and $\om$ be an open subset of $N$, such that $(\om,T)$ satisfies VGCC but not GCC. Then the Schr\"odinger equation is observable on $(\om,T)$, while GCC is not satisfied.
\end{corollary}

This result provides new examples of configurations where one has observability but not GCC.

%
%

\paragraph{Quantum limits on a product manifold.}
The definition of a quantum limit is recalled in Appendix \ref{ql}.

\begin{corollary} \label{cor3} 
 The support of any quantum limit of $N$ must contain at least an horizontal and a vertical geodesic. In particular, the Dirac measure along any periodic geodesic of $N$ is not a quantum limit. 
\end{corollary}


\section{Proofs}
\subsection{Proof of Theorem \ref{main}} 
Let $\om$ be an open subset of $N$. For any $x \in M$, we set $\om_x= \om \cap \left( \{x\} \times M'\right)$. Theorem \ref{main} follows from the following lemmas, which are in order.
 
 \begin{lemma} \label{lemma_main} 
Assume that there exists $c,T>0$ such that for all complex numbers $(a_{k,m})_{k,m  \in \N}$ and every $x \in M$,
\begin{equation} \label{main_estimate}
\int_0^T  \int_{\om_x }  \Big|\sum_{k,m} a_{k,m} \phi'_m e^{i F(\la_k + \mu'_m)t}  \Big|^2 \, dx_{h'}\, dt \geq c \sum_{k,m} |a_{k,m}|^2  
\end{equation}
then \eqref{schrod} is observable on $(\om,T)$.
\end{lemma}

\begin{proof}
The objective is to prove \eqref{obs}.
Writing $y  = \sum_{l,m \geq 0} b_{l,m} \phi_l \phi'_m$, we have $e^{i t F(\triangle_g)} y = \sum_{k,m} b_{k,m}  \phi_k \phi'_m e^{i F(\mu_k + \mu'_m)t}$.
We denote by $G_x : C^\infty(N) \to C^{\infty}(\{ x \} \times M')$ the mapping $(G_xf)(q,q')=f(x,q')$. Setting $a_{k,m}(x) = \sum_{l = \al_k}^{\al_{k+1}-1} b_{l,m} \phi_l(x)$, using (\ref{main_estimate}) and the definition of $\al_k$, there exists $T,c >0$ such that 
\begin{multline*}
\int_0^T \int_\om | e^{it F(\triangle_g)} y|^2 \,dx_g\, dt =  \int_0^T  \int_M \int_{\om_x} |G_x e^{it F(\triangle_g)} y|^2 \,dx_{h'} \,dx_h(x) \,dt \\
=  \int_0^T \int_M \int_{\om_x }  \Big|\sum_{l,m \geq 0} b_{l,m} \phi_l(x) \phi'_m(x') e^{i F(\mu_k + \mu'_m)t}  \Big|^2  \,dx_{h'}(x')\, dx_h(x) \,dt \\
= \int_0^T  \int_M \int_{\om_x } 
  \Big|\sum_{k,m} a_{k,m}(x) \phi'_m(x') e^{i F(\la_k + \mu'_m)t}\Big|^2 \, dx_{h'}(x')\, dx_h(x)\, dt \\
\geq c \int_0^T  \int_M \sum_{k,m} |a_{k,m}(x)|^2 \,dx_h(x) \,dt 
= c T \sum_{k,m}  \int_M |\sum_{l = \al_k}^{\al_{k+1}-1} b_{l,m} \phi_l|^2 dx_h  \\
= cT \sum_{k,m} \int_M \sum_{\al_k \leq l,l' \leq \al_{k+1}-1} b_{l,m} b_{l',m} \phi_l \phi_l' \,dx_h  
=  cT \sum_{k,m} \int_M  \sum_{l = \al_k}^{\al_{k+1}-1} b_{l,m}^2 \phi_l^2 \,dx_h \\
=  cT \sum_{k,m} b_{k,m}^2 \int_M \phi_k^2 \,dx_h 
= cT  \sum_{k,m} b_{k,m}^2 
= cT \|y\|_{L^2(N)}^2.
\end{multline*}
This proves observability in time $T$.
\end{proof}

We define
$$g_1^V(\om)= \inf_{x, \phi'} \int_{\om_x} \phi'^2 $$
 where the infimum is taken over the set of all possible $x \in M$ and all possible eigenfunctions $\phi'$ of $\Delta_{h'}$ such that $\| \phi'\|_{L^2(M)}=1$. 
 
\begin{lemma}  \label{case2}
Assume that the family $(F(\la_k+\la'_m))_{k,m \in \N}$ satisfies the gap condition. Then \eqref{main_estimate} is satisfied with $c= g_1^V(\om)/2$ for $T$ large enough.
\end{lemma}
 
\begin{proof}
Define $\La_{k,m}= F(\la_k + \mu'_m)$. By assumption, there exists $C_0>0$ such that if $\La_{k,m} \not= \La_{k',m'}$, then
\begin{equation} \label{gap}
 |\La_{k,m} - \La_{k',m'}| \geq C_0.
\end{equation}
Let $T>0$ and  $\psi_T$ the characteristic  function of the interval $[0,2T]$. Its Fourier transform $\hat{\psi}_T$ is equal to 
$\hat{\psi}_T(\xi)=  \frac{e^{iT\xi}-1}{T \xi}$. Noting that $\hat{\psi}_T(0)=1$, we have
\begin{multline} \label{sum}
\int_{0}^{2T}  \int_{\om_x }  \Big|\sum_{k,m} a_{k,m} \phi'_m  \, e^{i F(\la_k + \mu'_m)t}  \Big|^2 \, dx_{h'} \, dt  \\ 
= \sum_{k,m,k',m'} a_{k,m}\overline{a_{k',m'}} \hat{\psi}_T(\La_{k,m} - \La_{k',m'}) \int_{ \om_x} \phi'_m \phi'_{m'} = A   + B 
\end{multline}
with
$$
A= \sum_{k,m} |a_{k,m}|^2 \int_{ \om_x} |\phi'_m|^2  , \qquad
B= \sum_{(k,m) \not= (k',m')}  a_{k,m}\overline{a_{k',m'}} \hat{\psi}_T(\La_{k,m} - \La_{k',m'}) \int_{ \om_x} \phi'_m \phi'_{m'}. 
$$
Using the gap condition \eqref{gap}, it follows from Montgomery-Vaughan inequality (see \cite{MV}) that
$|B| \leq \frac{2}{T C_0} A$. Hence, we obtain from \eqref{sum} that 
$$
\int_{0}^{2T} \int_{\om_x }  \Big|\sum_{k,m} a_{k,m} \phi'_m e^{i F(\la_k + \mu'_m)t}  \Big|^2 \, dx_{h'}\, dt 
\geq \left(1 - \frac{2}{TC} \right)A
\geq \frac{1}{2} A
$$
when $T$ is large enough.
Noting that $A \geq \sum_{k,m} |a_{k,m}|^2 g_1^V(\om)$, the inequality \eqref{main_estimate} follows with $c= g_1^V(\om)/2$. 
\end{proof}

\begin{lemma} \label{g1vgcc}
If $(\om,T)$ satisfies VGCC then $g_1^V(\om)>0$.
\end{lemma}

\begin{proof}
Assume that $\om$ satisfies VGCC. By contradiction, let us assume that $g_1^V(\om)=0$. This means that for every $\ep >0$, there exists $x_\ep \in M$ and an eigenfunction $\phi'_\ep$ of $\Delta_{h'}$ such that   $\|\phi'_\ep\|_{L^2(M')} = 1$ and such that $\int_{\om_{x_\ep}} \phi_{\ep}^2 dx_g \leq \ep$, where we recall that  $\om_{x_\ep} = \om \cap \left(\{x_\ep \} \times M'\right)$.
By compactness, we assume that $x_\ep \to x_0 \in M$ and that $(\phi'_\ep)^2 \to \mu$ weakly, where $\mu$ is a quantum limit of $M'$. Let $U_k$ be an increasing sequence of open sets such that $\overline{U_k} \subset U_{k+1}$ and such that $\cup_k U_k = \om_0= \om \cap \left(\{x_0\} \times M'\right) $. Since $\om$ is open, for all $k\in\N$ and $\ep>0$ small enough, we have $U_k \subset \om_{x_\ep}$. This implies that 
$\int_{U_k} (\phi'_{\ep})^2 \, dx_g \leq \ep.$ We infer from the Portmanteau theorem (see Appendix \ref{cintre}) that $\mu(U_k)=0$, and thus $\mu(\om_0) =0$. 
This implies that GCC does not hold for $\om_0$ in any time. 
Indeed, by the Egorov theorem (see \cite{Egorov,zworski}), $\mu$ is invariant under the geodesic flow, as a measure on $S^*M'$. By the Krein-Milman theorem, $\mu$ can be approximated by a sequence $(\mu_k)_{k\in\N}$ of convex combinations of Dirac measures along periodic geodesics. Since $\mu_k(\om_0)\rightarrow 0$, there exists a sequence of periodic geodesics $\gamma_k$ such  that, if $\delta_k$ is the Dirac measure along $\gamma_k$, we have $\delta_k(\mu_0)\rightarrow 0$. This means that the time spent by $\gamma_k$  (actually, by its projection onto $M'$) in $\om_0$ tends to $0$. By compactness of geodesics, $\gamma_k$ converges to some geodesic $\gamma$. Again by the Portmanteau theorem, $\gamma$ does not meet $\om$, hence GCC on $M'$ fails for $\om_0$ and this contradicts that VGCC is satisfied for $\om$.  
\end{proof}

\subsection{Proof of Corollary \ref{cor3}}
We prove the vertical case, the horizontal case being symmetric.
We rearrange the set $\{ \la_j+\la'_k\ \mid\ j,k \in\N \} = \{ d_k\ \mid\ k \in\N\}$ with an increasing sequence $(d_k)_{k \in\N}$. Let $F$ be an increasing function such that $F(d_k)=k$ for every $k\in\N$.
By construction, the set $\{ F(\la_j+\la'_k) \ \mid\ j,k \in\N \}$ satisfies the gap condition. Let  $\Gamma$ be the support of a quantum limit $\mu$ on $M \times M'$. Since $F(\triangle_g)$ and $\triangle_g$ have the same eigenfunctions, $\mu$ is also the weak limit of a sequence of $\psi_j^2 \, dx_g\, d_\xi$ where $\psi_j$ are eigenfunctions of $F(\triangle_g)$ satisfying $\|\psi_j\|_{L^2} = 1$. We set $\om_\ep = \{x\in N\ \mid\ d_g(x,\Gamma)>\ep\}$, for $\ep>0$ small enough. For every $T>0$, $(\om_\ep,T)$ is not observable for \eqref{schrod} because $y=\psi_j$ provides a sequence of test functions which, at the limit, lie on $\Gamma$. Hence, by Theorem \ref{main}, $(\om_\ep,T)$ does not satisfy VGCC. Remark \ref{ex_vgcc}  implies that $\Gamma$ must contain a vertical geodesic.

\appendix

\section{Appendix}

\subsection{Quantum limits} \label{ql}
We recall that a \textit{quantum limit} (QL in short) $\mu$, also called \textit{semi-classical measure}, is a probability Radon (i.e., probability Borel regular) measure on $S^*M$ that is a closure point (weak limit), as $\lambda\rightarrow+\infty$, of the family of Radon measures $\mu_\lambda(a)=\langle\Op(a)\phi_\lambda,\phi_\lambda\rangle$ (which are asymptotically positive by the G\aa rding inequality), where $\phi_\lambda$ denotes an eigenfunction of norm $1$ associated with the eigenvalue $\lambda$ of $\sqrt{\triangle}$. Here, $\Op$ is any quantization.
We speak of a \textit{QL on $M$} to refer to a closure point (for the weak topology) of the sequence of probability Radon measures $\phi_\lambda^2\, dx_g$ on $M$ as $\lambda\rightarrow +\infty$. 
Note that QLs do not depend on the choice of a quantization. We denote by $\QL(S^*M)$ (resp., $\QL(M)$) the set of QLs (resp., the set of QLs on $M$). Both are compact sets.

Given any $\mu\in\QL(S^*M)$, the Radon measure $\pi_*\mu$, image of $\mu$ under the canonical projection $\pi:S^*M\rightarrow M$, is a probability Radon measure on $M$. It is defined, equivalently, by $(\pi_*\mu)(f) = \mu(\pi^*f) = \mu(f\circ\pi)$ for every $f\in C^0(M)$ (note that, in local coordinates $(x,\xi)$ in $S^*M$, the function $f\circ\pi$ is a function depending only on $x$), or by $(\pi_*\mu)(\omega)=\mu(\pi^{-1}(\omega))$ for every $\omega\subset M$ Borel measurable (or Lebesgue measurable, by regularity).
It is easy to see that\footnote{Indeed, given any $f\in C^0(M)$ and any $\lambda\in\mathrm{Spec}(\sqrt{\triangle})$, we have
$$
(\pi_*\mu_\lambda)(f) = \mu_\lambda(\pi^*f) = \langle\Op(\pi^*f)\phi_\lambda,\phi_\lambda\rangle = \int_M f \phi_\lambda^2\, dx_g ,
$$
because $\Op(\pi^*f)\phi_\lambda=f\phi_\lambda$. The equality then easily follows by weak compactness of probability Radon measures.
}
\begin{equation*}
\pi_* \QL(S^*M) = \QL(M) .
\end{equation*}
In other words, QLs on $M$ are exactly the image measures under $\pi$ of QLs.


\subsection{Portmanteau theorem} \label{cintre}
Let us recall the Portmanteau theorem (see, e.g., \cite{Billingsley}).
Let $X$ be a topological space, endowed with its Borel $\sigma$-algebra. 
Let $\mu$ and $\mu_n$, $n\in\N^*$, be finite Borel measures on $X$. Then the following items are equivalent:
\begin{itemize}
\item $\mu_n\rightarrow\mu$ for the narrow topology, i.e., $\int f\, d\mu_n\rightarrow\int f\, d\mu$ for every bounded continuous function $f$ on $X$;
\item $\int f\, d\mu_n\rightarrow\int f\, d\mu$ for every Borel bounded function $f$ on $X$ such that $\mu(\Delta_f)=0$, where $\Delta_f$ is the set of points at which $f$ is not continuous;
\item $\mu_n(B)\rightarrow\mu(B)$ for every Borel subset $B$ of $X$ such that $\mu(\partial B)=0$;
\item $\mu(F)\geq\limsup\mu_n(F)$ for every closed subset $F$ of $X$, and $\mu_n(X)\rightarrow\mu(X)$;
\item $\mu(O)\leq\liminf\mu_n(O)$ for every open subset $O$ of $X$, and $\mu_n(X)\rightarrow\mu(X)$.
\end{itemize}


\paragraph{Acknowledgment.} 
The first author is supported by the project THESPEGE (APR IA), R\'egion Centre-Val de Loire, France, 2018-2020.


\begin{thebibliography}{10} 

\small

\bibitem{ananth}
N. Anantharaman, M. L\'eautaud, F. Maci\`a,
\textit{Wigner measures and observability for the Schr\"odinger equation on the disk},
Invent. Math. {\bf 206} (2016), no. 2, 485--599. 

\bibitem{BardosLebeauRauch}
C. Bardos, G. Lebeau, J. Rauch,
\textit{Sharp sufficient conditions for the observation, control, and stabilization of waves from the boundary},
SIAM J. Control Optim. {\bf 30} (1992), no. 5, 1024--1065.


\bibitem{Billingsley}
P. Billingsley, 
\textit{Convergence of Probability Measures},
2nd ed., Wiley, 1999.



\bibitem{Egorov}
Y.  Egorov,
\textit{The  canonical  transformations  of  a  pseudo-differential  operator},
Uspehi. Mat. Nauk. {\bf 24} (1969), 235--236.








\bibitem{hpt} 
E. Humbert, Y. Privat, E. Tr\'elat,
\textit{Observability properties of the homogeneous wave equation on a closed manifold},
Comm. Partial Differential Equations {\bf 44} (2019), no. 9, 749--772.



\bibitem{jaffard}
S. Jaffard, 
\textit{Contr\^ole interne exact des vibrations dÕune plaque rectangulaire},
Portugal. Math. {\bf 47} (1990), 423--429.


\bibitem{Komornik} 
V. Komornik, P. Loreti, 
\textit{Fourier Series in Control Theory},
Springer-Verlag, New York, 2005.

%
\bibitem{lebeau}
G. Lebeau,
\textit{Contr\^ole de l'equation de Schr\"odinger},
J. Math. Pures Appl. {\bf 71} (1992), 267--291.




\bibitem{Macia_2011}
F. Maci\`a,
\textit{The Schr\"odinger flow in a compact manifold: high-frequency dynamics and dispersion},
in: Modern Aspects of the Theory of Partial Differential Equations, volume 216 of Operator Theory: Advances and Applications, pp. 275--289. Birkh\"auser/Springer Basel AG, Basel (2011).



\bibitem{MV}
H. L. Montgomery, R. C. Vaughan,
\textit{Hilbert's inequality},
J. London Math. Soc. {\bf 2} (1974), no.8, 73--82.









\bibitem{rauch-taylor} 
J. Rauch, M. Taylor,
\textit{Exponential decay of solutions to hyperbolic equations in bounded domains},
Indiana Univ. Math. J. {\bf 24} (1974), 79--86.   


\bibitem{zworski}      
M. Zworski.
\textit{Semiclassical Analysis}, 
Graduate Studies in Mathematics, Vol. 38, AMS, 2012.


\end{thebibliography}
\end{document}